\documentclass[12pt,reqno]{amsart}

\usepackage[margin=1in]{geometry}

\usepackage{graphicx}%
\usepackage{multirow}%

\usepackage{amsaddr}

\usepackage{commath}
\usepackage{amsfonts}
\usepackage{amsmath}
\usepackage{setspace}

\usepackage{amsthm}%

\usepackage{hyperref}

\usepackage{mathrsfs}%
\usepackage[title]{appendix}%
\usepackage{xcolor}%
\usepackage{textcomp}%
\usepackage{manyfoot}%
\usepackage{booktabs}%
\usepackage{algorithm}%
\usepackage{algorithmicx}%
\usepackage{algpseudocode}%
\usepackage{listings}%

\newcommand{\End}{\mathrm{End}}

\newcommand{\U}{\mathrm{U}}
\newcommand{\I}{\mathrm{I}}
\newcommand{\diag}{\mathrm{diag}}
\newcommand{\innerpro}[2]{\langle #1, #2 \rangle}

\newtheorem{theorem}{Theorem}[section]
\newtheorem{proposition}{Proposition}[section]
\newtheorem{lemma}{Lemma}[section]

\newtheorem{corollary}{Corollary}[section]

\newtheorem{definition}{Definition}[section]

\usepackage{fancyhdr}

\begin{document}
\title[Toeplitz $C^*$-algebras on Weighted Fock Spaces]{Toeplitz $C^*$-algebras on Radially Weighted Fock Spaces: Commutativity and Spectral Representation}

\author[K. Bdarneh]{%
Khalid Bdarneh\\[2pt]
{\scriptsize Department of Mathematics, University of Jordan}\\
{\scriptsize Amman, Jordan}\\
{\scriptsize \texttt{\lowercase{K.bdarneh@ju.edu.jo}}}
}

\begin{abstract}
    We study Toeplitz operators acting on radial weighted Fock spaces. We use tools from representation theory to construct commutative families of $C^*$-algebras that are generated by Toeplitz operators whose symbols are invariant under the action of $\U(n)$. For a partition $m=(b_1,...,b_k)$ of an integer $n$, we realize \[\mathbf{\U_m}:=\U(b_1)\times...\times\U(b_k)\] as a block diagonally subgroup of $\U(n)$ and describe the decomposition of the weighted Fock space into irreducible $\mathbf{U_m}$-modules. This allows us to study Toeplitz operators with $\mathbf{U_m}$-invariant, equivalently $k$-quasi-radial, symbols. Also, we provide an explicit integral representation of their eigenvalues. More generally, for an arbitrary compact subgroup $H\subseteq\U(n)$, we characterize the commutativity of the $C^*$-algebra generated by $H$-invariant Toeplitz operators in terms of the multiplicity-free property of the representation $\pi|_H$. Finally, for a logarithmically growing radial weight, we construct a bounded radial symbol for which the corresponding eigenvalue sequence is not uniformly continuous with respect to the square-root metric. Consequently, the uniform closure of the set of eigenvalue sequences does not coincide with the $C^*$-algebra of bounded sequences that are uniformly continuous with respect to the square-root metric.

\end{abstract}




\maketitle

\section{Introduction}
An interesting aspect of studying Toeplitz algebras focuses on characterizing the commutativity of the $C^*$-algebras generated by Toeplitz operators acting on various function spaces; see \cite{GRUDSKY20061,Raul2007commutative,grudsky2002toeplitz,Dawson2015}. In particular, suppose $T_\varphi$ is a Toeplitz operator with a symbol $\varphi$ acting on a certain function space, one may ask; what conditions do we need on the symbols to generate a commutative $C^*$-algebra? Furthermore, can we describe the spectrum of these Toeplitz operators in terms of their symbols?

Let $L^2(\mathbb{C},d\mu)$ be the space of all Lebesgue measurable functions on $\mathbb{C}$ with the Gaussian measure \[d\mu(z)=\dfrac{1}{\pi}e^{-|z|^2}dz\] where $dz$ is the Lebesgue measure on $\mathbb{C}$. The Fock space, denoted by $\mathcal{F}^2(\mathbb{C})$, is the subspace of $L^2(\mathbb{C},d\mu)$ consisting of entire functions. Let $P:L^2(\mathbb{C},d\mu)\rightarrow \mathcal{F}^2(\mathbb{C})$ be the orthogonal projection. Given a function $\varphi(z)\in L^{\infty}(\mathbb{C})$, the Toeplitz operator $T_\varphi:\mathcal{F}^2(\mathbb{C})\rightarrow \mathcal{F}^2(\mathbb{C})$ is defined by \[T_\varphi f(z)=P(f\varphi)(z)\]

Toeplitz operators on the Fock space have been studied for a long time; see \cite{bauer2015commuting,bauer2023self,coburn2011toeplitz,grudsky2002toeplitz,fulsche2024essential}. In general, the $C^*$-algebra generated by the Toeplitz operators with $L^{\infty}(\mathbb{C})$ symbols is not commutative. However, we can generate a commutative $C^*$-algebra by restricting the class of symbols. For example, the $C^*$-algebra generated by the Toeplitz operators with radial symbols is commutative \cite{grudsky2002toeplitz}. Moreover, it was shown in \cite{grudsky2002toeplitz} that if $\varphi(z)=a(|z|)$ is a radial symbol, then the Toeplitz operator $T_\varphi$ is unitarily equivalent to a multiplication operator $\lambda_{\varphi}I$, where the eigenvalues $\lambda_\varphi$ are described by 
\[\lambda_\varphi(m)=\dfrac{1}{m!}\int_{\mathbb{R}^+}a(\sqrt{r})r^me^{-r}dr\, ,\, m\in\mathbb{Z}^+.\]

In this paper, we will consider Toeplitz operators on radial weighted Fock spaces over $\mathbb{C}^n$. Let $h:[0,\infty)\rightarrow [0,\infty)$ be a continuous function that satisfies
\begin{equation}\label{int}
    \int\limits_0^\infty r^{2m+2n-1}\,e^{-2h(r)}<\infty
\end{equation}
 for every $m\in\mathbb{N}_0$.

The function $h$ is extended to $\mathbb{C}^n$ by $h(z):=h(|z|)$. 

The space $L_h^2(\mathbb{C}^n)$ is defined as the set of all Lebesgue measurable functions on $\mathbb{C}^n$ that satisfy
\[\norm{f}_h^2=\int_{\mathbb{C}^n}|f(z)|^2e^{-2h(z)} d z<\infty\]
where $d z$ is the Lebesgue measure on $\mathbb{C}^n$.

Let $\mathcal{O}(\mathbb{C}^n)$ be the space of all entire functions on $\mathbb{C}^n$. The weighted radial Fock space is defined as \[ \mathcal{A}_h(\mathbb{C}^n):=L_h^2(\mathbb{C}^n)\cap \mathcal{O}(\mathbb{C}^n).\]
 This is a reproducing kernel Hilbert space. The inner product on $\mathcal{A}_h(\mathbb{C}^n)$ is denoted by $\innerpro{.}{.}_h$, and for $z\in\mathbb{C}^n$, the function $K_z$ represents the reproducing kernel of $\mathcal{A}_h(\mathbb{C}^n)$. Therefore, for any $f\in \mathcal{A}_h(\mathbb{C}^n)$ we have

\[f(z)=\innerpro{f}{K_z}_{h}=\int_{\mathbb{C}^n} f(w)\overline{K_z(w)} e^{-2h(w)}dw \]
The orthogonal Bergman projection $P:L_h^2(\mathbb{C}^n)\rightarrow \mathcal{A}_h(\mathbb{C}^n)$ is given by \[Pf(z)=\innerpro{f}{K_z}_h\] and is a bounded linear operator, where $K_z(w)=K(w,z)$.
Let $\varphi\in L^\infty(\mathbb{C}^n)$, we define the Toeplitz-like operator $T_\varphi: \mathcal{A}_h(\mathbb{C}^n)\rightarrow \mathcal{A}_h(\mathbb{C}^n)$ by 
\[T_\varphi f(z):= P(\varphi f)(z)=\int_{\mathbb{C}^n}\varphi(w)f(w) \overline{K_z(w)}e^{-2h(w)}dw\]
Since the Bergman projection is a bounded linear operator, we have $\norm{T_\varphi}_h\leq \norm{\varphi}_{\infty}$.

One of our main goals in this paper is to use some tools from representation theory to show that under certain conditions the $C^*$-algebra, that is generated by Toeplitz operators that are acting on the radial weighted Fock space $\mathcal{A}_h(\mathbb{C}^n)$ is commutative. Let $\pi:H\rightarrow \mathcal{U}(\mathcal{H})$ be a unitary representation of a Lie group $H$ on a Hilbert space $\mathcal{H}$, and let $\End_H(\mathcal{H})$ be the set of all bounded operators on $\mathcal{H}$ that intertwine with $\pi$. If $H$ is a Lie group of type I domain, then $\End_H(\mathcal{H})$ is commutative if and only if $\pi$ is multiplicity-free \cite{Kobayashi2005}. For example, consider the unitary representation of the group $\U(n)$ on the classical Fock space that is defined by \[ \pi: \U(n)\times \mathcal{F}^2(\mathbb{C}^n)\rightarrow \mathcal{F}^2(\mathbb{C}^n)\] \[ \pi(A)f(z)=f(A^{-1}.z) \]
This is a multiplicity-free representation. However, the Toeplitz operators with symbols invariant under the action of the group $\U(n)$ commute with representation $\pi$, so the $C^*$-algebra that is generated by this type of Toeplitz operators is a subalgebra of $\End_{\U(n)}(\mathcal{F}^2(\mathbb{C}^n))$, and the commutativity of this $C^*$-algebra follows from the commutativity of $\End_{\U(n)}(\mathcal{F}^2(\mathbb{C}^n))$.

In addition to the study of the $C^*$-algebras, we will provide an explicit integral representation of the eigenvalues of the Toeplitz operators with symbols satisfying a certain invariance property. In fact, if a Toeplitz operator intertwines with an irreducible representation, then it is a scalar multiple of the identity operator. This will allow us to provide a spectral representation of the Toeplitz operators that act on $\mathcal{A}_h(\mathbb{C}^n)$.

 In the classical Fock space, Esmeral and Maximenko \cite{esmeral2016radial} showed that the eigenvalue sequences associated with bounded radial symbols are uniformly continuous with respect to the square-root metric
\[
\rho(m,n)=|\sqrt{m}-\sqrt{n}|
\]
and that their uniform closure coincides with the $C^*$-algebra of all bounded sequences that are uniformly continuous with respect to this metric. This result was generalized to quasi-radial symbols on the classical Fock space in higher dimensions \cite{dewage2022toeplitz}. In this paper, we show that this result does not hold for arbitrary radial weights. More precisely, for the logarithmically growing weight
\[
h(r)=
\begin{cases}
0 & :0\leq r\leq 1 \\
\dfrac{1}{2}(\log r)^{3/2} & : 1<r
\end{cases}
\]
we construct a bounded radial symbol $\varphi$ whose eigenvalue sequence satisfies
\[
\lim_{m\rightarrow\infty}
\left|\lambda_{\varphi,m}-(-1)^m\right|=0
\]
Consequently, the corresponding sequence of eigenvalues is not uniformly continuous with respect to the square-root metric. This means that the uniform closure of the set of sequences of all eigenvalues does not coincide with the $C^*$-algebra of bounded sequences that are uniformly continuous with respect to the square-root metric.

\section{Notation}

Let $\mathbb{N}_0=\lbrace 0,1,2,3,...\rbrace$ be the set of nonnegative integers. Throughout the paper, we fix an integer $n\in\mathbb{N}$. Let $\mathbf{m}=(b_1,...,b_k)$ be a partition of the integer $n=b_1+...+b_k$. If $z\in\mathbb{C}^n$, then $z$ can be written as $z=(z_{(1)},...,z_{(k)})\in \mathbb{C}^{b_1}\times...\times \mathbb{C}^{b_k}$, where \[z_{(i)}=(z_{i1},...,z_{ib_i})\]

Let $z\in\mathbb{C}^n$. For every multi-index $r=(r_1,...,r_n)\in\mathbb{N}_0^n$ we will use the following notation 
\begin{align*}
    z^r&=z_1^{r_1}...z_n^{r_n},\\
    |r|&=r_1+...+r_n,\\
    r!&=r_1!...r_n!
\end{align*}

\section{Representation of $\U(n)$ and $\mathbb{T}^n$ on $\mathcal{A}_h(\mathbb{C}^n)$}
The group $\U(n)$ of all $n\times n$ unitary matrices acts on $\mathbb{C}^n$ by \[ \U(n)\times \mathbb{C}^n\rightarrow \mathbb{C}^n\] \[ (A,z)\mapsto A.z \]
This action induces an action of $\U(n)$ on $\mathcal{A}_h(\mathbb{C}^n)$ that is defined by $A.f(z)=f(A^{-1}.z)$. In fact, this action leaves the measure $e^{-2h(z)}.dz$ invariant.
Furthermore, this action gives rise to a unitary representation of $\U(n)$ on $\mathcal{A}_h(\mathbb{C}^n)$ that is defined by \[ \pi : \U(n)\times \mathcal{A}_h(\mathbb{C}^n) \rightarrow \mathcal{A}_h(\mathbb{C}^n) \] \[ \pi(A) f(z)=f(A^{-1}.z)\]

\begin{definition}
    Suppose that $G$ is a group acting on $\mathbb{C}^n$. Then $G$ acts on $L^{\infty}(\mathbb{C}^n)$ by $g.\varphi(z)=\varphi(g^{-1}.z)$. We say that $\varphi$ is $G$-invariant if $g.\varphi(z)=\varphi(z)$ for all $g\in G$. The class of all symbols in $L^{\infty}(\mathbb{C}^n)$ that are invariant under the action of $G$ is denoted by $L^{\infty}(\mathbb{C}^n)^{G}$.
\end{definition}

The following lemma shows how the group $\U(n)$ acts on the kernel of $\mathcal{A}_h(\mathbb{C}^n)$. We will need this lemma to show that Toeplitz operators with symbols that are invariant under the action of $\U(n)$ intertwine with the representation $\pi$.
\begin{lemma}
Let $A\in\U(n)$. Then the kernel $K_z(w)$ satisfies $ K_z(A.w)=K_{A^{-1}.z}(w)$.
\begin{proof}
    If $f\in \mathcal{A}_h(\mathbb{C}^n)$, then we have 
    \begin{align*}
        \innerpro{f}{\pi(A^{-1})K_z}_h&=\innerpro{\pi(A)f}{K_z}_h= \pi(A)f(z)\\
        &=f(A^{-1}z)\\
        &=\innerpro{f}{K_{A^{-1}z}}_h
    \end{align*}
    Since $f$ was arbitrary chosen, we have $\pi(A^{-1})K_z(w)=K_{A^{-1}z}(w)$. Therefore, $K_z(A.w)=K_{A^{-1}z}(w)$
\end{proof}
    
\end{lemma}
The following result shows how a Toeplitz operator with symbol $\varphi$ interacts with the representation $\pi$.
\begin{proposition} \label{Toeplitz and rep}
    Let $\varphi\in L^{\infty}(\mathbb{C}^n)$. We have \[T_{A.\varphi}\pi(A)=\pi(A)T_{\varphi}\]

     \begin{proof}
        \begin{align*}
             T_{A.\varphi}\pi(A)f(z)=T_{A.\varphi}f(A^{-1}.z)&=\int (A.\varphi)(w) f(A^{-1}.w)\overline{K_z(w)}e^{-2h(w)}dw\\
             &= \int \varphi(A^{-1}.w) f(A^{-1}.w)\overline{K_z(w)}e^{-2h(w)}dw\\
             &= \int \varphi(w) f(w)\overline{K_{z}(A.w)}e^{-2h(A.w)}dw\\
             &= \int \varphi(w) f(w)\overline{K_{A^{-1}.z}(w)}e^{-2h(w)}dw\\
             &= \pi(A) T_{\varphi}f(z)
        \end{align*}
    \end{proof}
\end{proposition}

We now have the following corollary:
\begin{corollary}
    
\label{Toeplitz operator intertwine with the representation}
    Let $\varphi\in L^{\infty}(\mathbb{C}^n)^{U(n)}$. The Toeplitz operator $T_{\varphi}$ intertwines with the unitary representation $\pi$. 
    
\end{corollary}

We realize the group $\mathbb{T}^n$ as a subgroup of $\U(n)$ by viewing its elements as diagonal matrices in $\U(n)$. In fact, the restriction $\pi|_{\mathbb{T}^n}$ is a unitary representation of $\mathbb{T}^n$ on $\mathcal{A}_h(\mathbb{C}^n)$.

We now have the following result:
\begin{corollary}
    Let $\varphi\in L^{\infty}(\mathbb{C}^n)^{\mathbb{T}^n}$. The Toeplitz operator $T_{\varphi}$ intertwines with the unitary representation $\pi|_{\mathbb{T}^n}$. 
    
\end{corollary}
\section{Decomposition of $\mathcal{A}_h(\mathbb{C}^n)$}
Let $\mathcal{P}(\mathbb{C}^n)$ denote the space of holomorphic polynomials on $\mathbb{C}^n$, and $P^m( \mathbb{C}^n )$ denote the subspace of homogeneous polynomials of degree $m$. Each $\mathcal{P}^m(\mathbb{C}^n)$ is an invariant irreducible $U(n)$-module. The highest weight of $P^{m}(\mathbb{C}^n)$ is \[(0,0,...,0,-m)\] Hence,
if $m_1\neq m_2$, then $P^{m_1}( \mathbb{C}^n )$ and $P^{m_2}( \mathbb{C}^n )$ have different highest weights and therefore they are inequivalent $U(n)$-representations. So, the representation $\pi$ decomposes into inequivalent irreducible subrepresentations. Therefore, $\pi$ is multiplicity-free. Moreover, since the holomorphic polynomials are dense in $\mathcal{A}_h(\mathbb{C}^n)$, and the homogeneous polynomials are mutually orthogonal, we have the isotypic decomposition of the representation $\pi$ of $\U(n)$ \[ \mathcal{A}_h(\mathbb{C}^n)=\bigoplus_{m=0}^{\infty} P^m( \mathbb{C}^n )\]

 If $A=\diag(d_1,...,d_n)\in\mathbb{T}^n$, then $\pi(A)z^m=(A^{-1}z)^m=A^{-m}z^m$. So, the isotypic decomposition of $\pi|_{\mathbb{T}^n}$, the restriction of the representation $\pi$ to $\mathbb{T}^n$, is given by \[\mathcal{A}_h(\mathbb{C}^n)=\bigoplus_{m\in\mathbb{N}_0}\mathbb{C}z^{m}\] Since the irreducible subrepresentations $\mathbb{C}z^{m}$ are inequivalent for different values of $m\in\mathbb{N}_0$, then $\pi|_{\mathbb{T}^n}$ is multiplicity-free.

The following two results summarize the previous discussion; we will need them to construct commuting families of $C^*$-algebras.

\begin{proposition} \label{Isotypic decomposiyion under U(n)}
     The representation $\pi$ of the group $\U(n)$ on $\mathcal{A}_h(\mathbb{C}^n)$ is multiplicity-free, and the isotypic decomposition is given by \[\mathcal{A}_h(\mathbb{C}^n)=\bigoplus_{m=0}^{\infty} P^m( \mathbb{C}^n ) \]
\end{proposition}

\begin{proposition}
     The representation $\pi|_{\mathbb{T}^n}$ of the group $\mathbb{T}^n$ on $\mathcal{A}_h(\mathbb{C}^n)$ is multiplicity-free, and the isotypic decomposition is given by \[\mathcal{A}_h(\mathbb{C}^n)=\bigoplus_{m\in\mathbb{N}_0^n}\mathbb{C}z^{m} \]
\end{proposition}

Let $\End_{\U(n)}(\mathcal{A}_h(\mathbb{C}^n))$ be the set of all bounded operators on $\mathcal{A}_h(\mathbb{C}^n)$ that intertwine with the representation $\pi$. The space $\End_{\U(n)}(\mathcal{A}_h(\mathbb{C}^n))$ is commutative if and only if the representation $\pi$ is multiplicity-free \cite{Kobayashi2005}. The $C^*$-algebra generated by the Toeplitz operators with $\U(n)$ invariant symbols will be denoted by $\mathcal{T}(\mathcal{A}_h(\mathbb{C}^n))^{\U(n)}$. The following proposition shows that this $C^*$-algebra is commutative.

\begin{proposition}
    The $C^*$-algebra $\mathcal{T}(\mathcal{A}_h(\mathbb{C}^n))^{\U(n)}$ that is generated by Toeplitz operators with symbols invariant under $\U(n)$ is commutative.
    \begin{proof}
       The proof follows directly from Corollary \ref{Toeplitz operator intertwine with the representation}. If a symbol $\varphi$ is invariant under the action of $\U(n)$, then the Toeplitz operator $T_{\varphi}$ intertwines with the representation $\pi$.
       Therefore, \[ \mathcal{T}(\mathcal{A}_h(\mathbb{C}^n))^{\U(n)}\subset End_{\U(n)}(\mathcal{A}_h(\mathbb{C}^n)).\]
       Since $End_{\U(n)}(\mathcal{A}_h(\mathbb{C}^n))$ is commutative, we have $\mathcal{T}(\mathcal{A}_h(\mathbb{C}^n))^{\U(n)}$ is also commutative.
    \end{proof}
\end{proposition}

A similar argument can be used to prove the following result.

\begin{proposition}
    The $C^*$-algebra $\mathcal{T}(\mathcal{A}_h(\mathbb{C}^n))^{\mathbb{T}^n}$ that is generated by Toeplitz operators with symbols invariant under $\mathbb{T}^n$ is commutative.
\end{proposition}

\section{$T_{\varphi}$ with $k$-Quasi-Radial Symbols}

    Let $\mathbf{m}=(b_1,...,b_k)$ be a partition of an integer $n=b_1+...+b_k$. For every $z\in\mathbb{C}^n$ we write $z=(z_{(1)},...,z_{(k)})\in \mathbb{C}^{b_1}\times...\times \mathbb{C}^{b_k}$. Consider the group $\mathbf{U_m}:=U(b_1)\times...\times U(b_k)$. This group can be realized as a block diagonal subgroup of $U(n)$ with action on $f(z)$ given by \[(A_1,...,A_k).f(z)=f(A_1^{-1}z_{(1)},...,A_k^{-1}z_{(k)})\]
\begin{definition}[\cite{quiroga2021toeplitz}]
    Let $\mathbf{m}=(b_1,...,b_k)$ be a partition of an integer $n=b_1+...+b_k$. A symbol $\varphi\in L^{\infty}(\mathbb{C}^n)$ is called a $k$-quasi-radial symbol if $\varphi$ is invariant under the action of the group $\mathbf{U_m}$.
\end{definition}

Let $\mathbf{s}:=(s_1,...,s_k)\in\mathbb{N}_0^k$. For each $i=1,2,...,k$, the space $P^{s_i}(\mathbb{C}^{b_i})$ of homogeneous holomorphic polynomials of degree $s_i$ is an irreducible $\U(b_i)$-module. Let $\pi_{b_i,s_i}$ be the unitary representation of $\U(b_i)$ on $P^{s_i}(\mathbb{C}^{b_i})$. Consider the outer tensor product \[\pi_{\mathbf{m,s}}:=\pi_{b1,s_1}\otimes...\otimes \pi_{b_k,s_k}\] which is a representation of $\mathbf{U_m}$ on 
\[P^{s_1}(\mathbb{C}^{b_1})\otimes...\otimes P^{s_k}(\mathbb{C}^{b_k}).\]
The representation is defined by
\begin{equation}
\label{representation of U_m}
    (\pi_{\mathbf{m,s}}(A_1,...,A_k)f)(z)=f(A_1^{-1}z_{(1)},...,A_k^{-1}z_{(k)})
\end{equation}

Since each $\pi_{b_i,s_i}$ is irreducible, their outer tensor product $\pi_{\mathbf{m,s}}$ is an irreducible representation of $U(b_1)\times...\times U(b_k)$. We also consider the representation $\pi_{\mathbf{m}}$ of $\mathbf{U_m}$ acting on the weighted Fock space $\mathcal{A}_h(\mathbb{C}^n)$, where the action is also given by Equation \ref{representation of U_m}.

The following proposition provides the isotypic decomposition of the representation $\pi_{\mathbf{m}}$. Moreover, it shows that this representation is multiplicty free. We will need this result later to construct commuting families of $C^*$-algebras.

\begin{proposition}
\label{Isotypic decomposition of U_m}
    The representation $\pi_{\mathbf{m}}$ of the group \[\mathbf{U_m}=U(b_1)\times...\times U(b_k)\] on $\mathcal{A}_h(\mathbb{C}^n)$ is multiplicity-free, and the isotypic decomposition is given by \[ \mathcal{A}_h(\mathbb{C}^n)=\bigoplus_{s_1,...,s_k=0}^{\infty} P^{s_1}(\mathbb{C}^{b_1})\otimes...\otimes P^{s_k}(\mathbb{C}^{b_k}) \] 

    \begin{proof}
        For each $i=1,2,...,k$, the space $P^{s_i}(\mathbb{C}^{b_i})$ is invariant under the action of the group $\U(b_i)$. Hence, the space $P^{s_1}(\mathbb{C}^{b_1})\otimes...\otimes P^{s_k}(\mathbb{C}^{b_k})$ is invariant under the action of the group $U(b_1)\times...\times U(b_k)$. Moreover, since the outer tensor product $\pi_{b1,s_1}\otimes...\otimes \pi_{b_k,s_k}$ is an irreducible representation and the monomials are dense in $\mathcal{A}_h(\mathbb{C}^n)$, the Hilbert space $\mathcal{A}_h(\mathbb{C}^n)$ decomposes into a direct sum of irreducible submodules.

        Now, we will show that the spaces $P^{s_1}(\mathbb{C}^{b_1})\otimes...\otimes P^{s_k}(\mathbb{C}^{b_k})$ are mutually orthogonal. Let $\mathbf{r,s}\in\mathbb{N}_0^k$ with $\mathbf{r}\neq\mathbf{s}$. There exists $j$ such that $s_j\neq r_j$. Let $t\in\mathbb{R}$ and consider the element 
        \[A=(I_{b_1},...,e^{it}I_{b_j},...,I_{b_k})\in \mathbf{U_m}\]
Let $f\in P^{s_1}(\mathbb{C}^{b_1})\otimes...\otimes P^{s_k}(\mathbb{C}^{b_k})$ and $g\in P^{r_1}(\mathbb{C}^{b_1})\otimes...\otimes P^{r_k}(\mathbb{C}^{b_k})$. We have
\begin{align*}
    \innerpro{f}{g}_h&=\innerpro{\pi_\mathbf{m}(A)f}{\pi_\mathbf{m}(A)g}_h\\
    &=\innerpro{e^{-its_j}f}{e^{-itr_j}g}_h\\
    &=e^{it(r_j-s_j)}\innerpro{f}{g}_h
\end{align*}
If we choose any $t\in\mathbb{R}$ with $e^{it(r_j-s_j)}\neq 1$, then $\innerpro{f}{g}_h=0$. Therefore, the spaces \[P^{s_1}(\mathbb{C}^{b_1})\otimes...\otimes P^{s_k}(\mathbb{C}^{b_k})\] are mutually orthogonal.

Furthermore, the highest weight of the irreducible $\mathbf{U_m}$-module \[P^{s_1}(\mathbb{C}^{b_1})\otimes...\otimes P^{s_k}(\mathbb{C}^{b_k})\]
is

\[ \gamma_{\mathbf{s}}:=\big( (0,0,...\,,0,-s_1),..., (0,0,...,0,\,-s_k)\big)\]
where the $i$th tuple has $b_i$ entries. If $\mathbf{s\neq r}$, then $\gamma_{\mathbf{s}}\neq \gamma_{\mathbf{r}}$. Hence, the subspaces \[P^{s_1}(\mathbb{C}^{b_1})\otimes...\otimes P^{s_k}(\mathbb{C}^{b_k}) \text{ and } P^{r_1}(\mathbb{C}^{b_1})\otimes...\otimes P^{r_k}(\mathbb{C}^{b_k})\] are  inequivalent $\mathbf{U_m}$-submodules. Thus, the decomposition of $\mathcal{A}_h(\mathbb{C}^n)$ consists of pairwise inequivalent irreducible $\mathbf{U_m}$-submodules, and therefore the representation $\pi_{\mathbf{m}}$ is multiplicity-free.

    \end{proof}
\end{proposition}

For the proof of the following proposition, it suffices to note that $e^{-2h(z)}dz_{(1)}...dz_{(k)}$ is invariant under the action of the group $U(b_1)\times...\times U(b_k)$, and that the kernel satisfies the relation \[K_z\left((A_1,...,A_k).(w_{(1)},...,w_{(k)})\right)=K_{\left( A_1^{-1}z_{(1)},...,A_k^{-1}z_{(k)}\right)}(w)\] Now the proof follows from an argument similar to the proof of Proposition \ref{Toeplitz operator intertwine with the representation}.
\begin{proposition}
\label{Toeplitz operator intertwine, quasi-radial}
    Let $\varphi\in L^{\infty}(\mathbb{C}^n)$. If $\varphi$ is invariant under the action of $U(b_1)\times...\times U(b_k)$, then $T_{\varphi}$ intertwines with the representation $\pi_{b1}\otimes...\otimes \pi_{b_k}$.

\end{proposition}

The following result generalizes Corollary 2.8 of \cite{dewage2022toeplitz}. It shows that the $C^*$-algebra generated by Toeplitz operators with symbols invariant under $U(b_1)\times...\times U(b_k)$ is commutative.  
\begin{proposition}
    Let $\mathbf{U_m} :=U(b_1)\times...\times U(b_k)$. Then the $C^*$-algebra $\mathcal{T}(\mathcal{A}_h(\mathbb{C}^n))^{\mathbf{U_m}}$, generated by Toeplitz operators with symbols that are invariant under $\mathbf{U_{m}}$, is commutative.

    \begin{proof}
    By Proposition \ref{Isotypic decomposition of U_m}, the representation $\pi_{b1}\otimes...\otimes \pi_{b_k}$ of the group $\mathbf{U_m}$ is multiplicity-free. So, the space $End_{\mathbf{U_m}}(\mathcal{A}_h(\mathbb{C}^n))$ is commutative. If $\varphi$ is $\mathbf{U_m}$ invariant, then the Toeplitz operator $T_\varphi$ intertwines with the representation $\pi_{b1}\otimes...\otimes \pi_{b_k}$. So, we have  \[ \mathcal{T}(\mathcal{A}_h(\mathbb{C}^n))^{\mathbf{U_m}}\subset \End_{\mathbf{U_m}}(\mathcal{A}_h(\mathbb{C}^n)). \] Therefore, the $C^*$-algebra $\mathcal{T}(\mathcal{A}_h(\mathbb{C}^n))^{\mathbf{U_m}}$ is commutative.
      
    \end{proof}
\end{proposition}

\section{Spectral representation of Toeplitz operators}
In this section we will describe the spectral representation of Toeplitz operators with symbols that are invariant under $\U(n)$ and $\mathbf{U_m}$. In particular, we will use Schur's lemma to write the spectrum of a Toeplitz operator $T_\varphi$ as an integral formula that involves the symbol $\varphi$.

\begin{theorem}[Schur's Lemma]

Let $\pi$ be a unitary representation of a group $G$, and let $\mathcal{C}(\pi)$ be the space of bounded operators on $\mathcal{H}_{\pi}$ that intertwine with $\pi$. The representation $\pi$ is irreducible if and only if $\mathcal{C}(\pi)$ consists only of scalar multiples of the identity.

\end{theorem}

The isotypic decomposition corresponding to the representation of $\U(n)$ on $\mathcal{A}_h(\mathbb{C}^n)$, Proposition \ref{Isotypic decomposiyion under U(n)}, is given by \[\mathcal{A}_h(\mathbb{C}^n)=\bigoplus_{m=0}^{\infty} P^m( \mathbb{C}^n ) \]
This means that if $\Psi\in \End_{\U(n)}(\mathcal{A}_h(\mathbb{C}^n))$, then we have \[\Psi=\bigoplus_{m=0}^\infty \Psi|_{P^m( \mathbb{C}^n )}\]
Hence, by Schur's Lemma, there exists $\lambda_m\in\mathbb{C}$ such that \[\Psi|_{P^m( \mathbb{C}^n )}=\lambda_m I_{P^m( \mathbb{C}^n )}\] where $I_{P^m( \mathbb{C}^n )}$ is the identity operator on $P^m( \mathbb{C}^n )$. 

As a consequence of the above discussion, we have the following result.

\begin{proposition}
   For every $\Psi\in \End_{\U(n)}(\mathcal{A}_h(\mathbb{C}^n))$, there exists $\lambda_m\in\mathbb{C}$ such that \[\Psi=\bigoplus_{m\in\mathbb{N}_0} \lambda_m I_{P^m( \mathbb{C}^n )}\] Moreover, the map defined by
    \begin{align*}
        \End_{\U(n)}(\mathcal{A}_h(\mathbb{C}^n)) \rightarrow \ell_{\infty}(\mathbb{N}_0)\\ \Psi\mapsto (\lambda_m)_{m\in\mathbb{N}}
    \end{align*} is an isomorphism of $C^*$-algebras.
\end{proposition}

The following well-known lemma presents a characterization of $\U(n)$-invariant symbols.

\begin{lemma}
    A symbol $\varphi$ is $U(n)-$ invariant if and only if there exists a function $\textbf{a}_{\varphi}:\mathbb{R}^{+}\rightarrow \mathbb{C}$ such that $\varphi(z)=\textbf{a}_{\varphi}(|z|)$ for almost all $z\in\mathbb{C}^n$.
\end{lemma}

The integration of monomials can be done nicely using the following classic result \cite{folland2001integrate}. We will need it to provide an elegant description of the eigenvalues $\lambda_{\varphi}$.
\begin{lemma}
   Let $\sigma$ denote the $(n-1)$-dimensional surface measure on $\mathbb{S}^{n-1}$, then the formula for integration in polar coordinates is \[ \int_{\mathbb{R}^n} f(x) dx=\int_{\mathbb{S}^n}\int_0^{\infty} f(\rho\,\omega) \rho^{n-1}d\rho \, d\sigma(\omega)\]
\end{lemma}

For the diagonalization of $T_\varphi$, we first consider symbols $\varphi$ that are invariant under the action of the group $\U(n)$. By Corollary \ref{Toeplitz operator intertwine with the representation}, every Toeplitz operator $T_\varphi$ on $\mathcal{A}_h(\mathbb{C}^n)$ whose symbol is $\U(n)$-invariant intertwines with representation $\pi$. This gives us the following result.

\begin{theorem}\label{eigenvalues radial}
    Let $\varphi\in L^{\infty}(\mathbb{C}^n)$ be $\U(n)$-invariant, and write $\varphi(z)= \mathbf{a}_\varphi(|z|)$ for almost every $z\in\mathbb{C}^n$. Then, for every $m\in\mathbb{N}_0$, \[T_\varphi|_{P^m(\mathbb{C}^n)}=\lambda_{\varphi,m}\I\]
    
    where the eigenvalues $\lambda_{\varphi,m}$ can be written as 

    \[\lambda_{\varphi,m}=\dfrac{\displaystyle\int_0^\infty \mathbf{a}_\varphi(\rho)\,\rho^{2m+2n-1}e^{-2h(\rho)}\, d\rho}{\displaystyle\int_0^\infty\, \rho^{2m+2n-1}e^{-2h(\rho)}\,d\rho}.\]

 \begin{proof}
Since $\varphi$ is $U(n)$-invariant, the Toeplitz operator
$T_\varphi$ intertwines with the representation $\pi$. By the isotypic decomposition of $\mathcal{A}_h(\mathbb{C}^n)$ and
Schur's lemma, for every $m\in\mathbb{N}_0$ there exists
$\lambda_{\varphi,m}\in\mathbb{C}$ such that
\[
\left.T_\varphi\right|_{P^m(\mathbb{C}^n)}
=
\lambda_{\varphi,m}I.
\]
 
 Let $r\in\mathbb{N}_0^n$ be such that $|r|=m$. Since $z^r\in P^m(\mathbb{C}^n)$, we have

 \[\innerpro{\lambda_{\varphi,m}z^r}{z^r}_h=\innerpro{T_{\varphi}z^r}{z^r}_h=\innerpro{\varphi z^r}{z^r}_h=\int_{\mathbb{C}^n} \varphi(z) |z|^{2r} e^{-2h(z)}dz\]
 Since $\varphi$ is $U(n)$-invariant. Then $\varphi(z)=\mathbf{a}_\varphi(|z|)$ for almost all $z\in\mathbb{C}^n$. Write $z=\rho\, \omega$ where $\rho>0$ and $\omega\in\mathbb{S}^n$.
 Therefore,
 \begin{align*}
     \lambda_{\varphi,m}&=\dfrac{\displaystyle\int_{\mathbb{C}^n} \varphi(z) |z|^{2r} e^{-2h(z)}dz}{\displaystyle\int_{\mathbb{C}^n} |z|^{2r} e^{-2h(z)}dz}\\
     &= \dfrac{\displaystyle\int_{\mathbb{S}^n}\int_0^\infty \mathbf{a}_\varphi(\rho) \rho^{2|r|} \,|\omega^r|^2 e^{-2h(\rho)} \rho^{2n-1}\, d\rho \, d\omega}{\displaystyle\int_{\mathbb{S}^n}\int_0^\infty \rho^{2|r|}| \,\omega^r|^2\,\rho^{2n-1}\,e^{-2h(\rho)}\,d\rho \, d\omega}\\
     &= \dfrac{\displaystyle\int_0^\infty \mathbf{a}_\varphi(\rho) \rho^{2m+2n-1}\, e^{-2h(\rho)} \, d\rho}{\displaystyle\int_0^\infty \rho^{2m+2n-1}\,e^{-2h(\rho)}\,d\rho}
 \end{align*}
     
 \end{proof}
\end{theorem}

Let $\mathbf{m}:=(b_1,...,b_k)\in \mathbb{N}_0^k$ be a partition of an integer $n=b_1+...+b_k$. We will consider the case where the symbol $\varphi$ is invariant under the group $\mathbf{U_m} :=U(b_1)\times...\times U(b_k)$. The isotypic decomposition of $\mathcal{A}_h(\mathbb{C}^n)$ under the action of $\mathbf{U_m}$, presented in Proposition \ref{Isotypic decomposition of U_m}, is given by \[ \mathcal{A}_h(\mathbb{C}^n)=\bigoplus_{s_1,...,s_k=0}^{\infty} P^{s_1}(\mathbb{C}^{b_1})\otimes...\otimes P^{s_k}(\mathbb{C}^{b_k}) \] 
Again, Schur's Lemma implies that if $\Psi\in \End_{\mathbf{U_m}}(\mathcal{A}_h(\mathbb{C}^n))$, then we have \[\Psi=\bigoplus_{s_1,...,s_k=0}^\infty \lambda_{(s_1,...,s_k)} \I|_{P^{s_1}(\mathbb{C}^{b_1})\otimes...\otimes P^{s_k}(\mathbb{C}^{b_k})}\]
where $\lambda_{(s_1,...,s_k)}\in\mathbb{C}$.

The following result is an immediate consequence of the previous discussion.

\begin{proposition}
    Let $\mathbf{m}=(b_1,...,b_k)$ be a partition of an integer $n=b_1+...+b_k$. For every $\Psi\in \End_{\mathbf{U_m}}(\mathcal{A}_h(\mathbb{C}^n))$, there exists $\lambda_{(s_1,...,s_k)}\in\mathbb{C}$ such that \[\Psi=\bigoplus_{s_1,...,s_k=0}^\infty \lambda_{(s_1,...,s_k)} \I|_{P^{s_1}(\mathbb{C}^{b_1})\otimes...\otimes P^{s_k}(\mathbb{C}^{b_k})}\] Moreover, the map defined by
    \begin{align*}
        \End_{\mathbf{U_m}}(\mathcal{A}_h(\mathbb{C}^n)) \rightarrow \ell_{\infty}(\mathbb{N}_0^k)\\ \Psi\mapsto (\lambda_s)_{s\in\mathbb{N}_0^k}
    \end{align*} is an isomorphism of $C^*$-algebras.
\end{proposition}

We recall the following lemma.

\begin{lemma}[\cite{dewage2022toeplitz}]
    A symbol $\varphi\in L^{\infty}(\mathbb{C}^n)$ is invariant under $\mathbf{\U_m}$ if and only if there exists a function $\mathbf{a}_{\varphi}:\mathbb{R}_{+}^k\rightarrow \mathbb{C}$ such that $\varphi(z)=\mathbf{a}_{\varphi}(|z_{(1)}|,...,|z_{(k)}|)$.
\end{lemma}

The following theorem describes the eigenvalues of $T_\varphi$ with a $k$-quasi-Radial symbol $\varphi$. The proof is similar to the proof of Theorem 2.12 in \cite{dewage2022toeplitz}.

\begin{theorem}
     Let $\mathbf{m}=(b_1,...,b_k)$ be a partition of an integer $n=b_1+...+b_k$, and assume that $\varphi\in L^{\infty}(\mathbb{C}^n)$ is invariant under $\mathbf{U_m}$. Then for every $\mathbf{s}=(s_1,\dots,s_k)\in \mathbb{N}_0^k$  \[T_\varphi|_{P^{s_1}(\mathbb{C}^{b_1})\otimes...\otimes P^{s_k}(\mathbb{C}^{b_k})}=\lambda_{\varphi,\mathbf{s}}\I.\] The eigenvalues $\lambda_{\varphi,\mathbf{s}}$ satisfy \[\lambda_{\varphi,\mathbf{s}}=\dfrac{\displaystyle\int_{\mathbb{R}_0^k}\, \mathbf{a}_{\varphi}(r_1,\dots,r_k) \prod_{j=1}^k r_j^{2s_j+2b_j-1}  e^{-2h\left(|\mathbf{r}|\right)}\, d\mathbf{r}}{\displaystyle\int_{\mathbb{R}_0^k}\,\, \prod_{j=1}^k r_j^{2s_j+2b_j-1}  e^{-2h\left(|\mathbf{r}|\right)}\, d\mathbf{r}}. \]
    where $d\mathbf{r}=dr_1\dots dr_k$ and $|\mathbf{r}|=\sqrt{r_1^2+\dots + r_k^2}$.
     \begin{proof}
         Let $\boldsymbol{\alpha}=(\alpha_1,...,\alpha_k)$ be such that $|\alpha_j|=s_j$. Let $p^{\alpha_j}(z)=z^{\alpha_j}$ and consider the monomial $q_{{\boldsymbol{\alpha}}}=p_{\alpha_1}(z_{(1)})...p_{\alpha_k}(z_{(k)})$ which belongs to $P^{s_1}(\mathbb{C}^{b_1})\otimes...\otimes P^{s_k}(\mathbb{C}^{b_k})$. Assume $|z_{(j)}|=r_j$. First, note that

         \begin{align*}
            \int_{\mathbb{C}^{b_j}} |p_{\alpha_j}(z_{(j)})|^2e^{-2h(|\mathbf{r}|)} \, d z_{(j)}=& \int_{\mathbb{S}^{2b_j-1}}\int_0^{\infty} |(r_j\,\omega_{(j)})^{\alpha_j}|^2 r_j^{2b_j-1} e^{-2h(|\mathbf{r}|)}dr_j \, d\sigma(\omega_{(j)})\\
             =& \int_{\mathbb{S}^{2b_j-1}} |\omega_{(j)}^{\alpha_j}|^2 d\sigma(\omega_{(j)}) \int_0^{\infty} r_j^{2s_j} r_j^{2b_j-1}e^{-2h(|\mathbf{r}|)}dr_j\\
             =& \dfrac{2\pi^{b_j} \alpha_j!}{(b_j-1+|\alpha_j|)!}\int_0^{\infty} r_j^{2b_j+2s_j-1}e^{-2h(|\mathbf{r}|)}dr_j
         \end{align*}

Moreover, 
         \begin{align*}
           &\innerpro{\varphi q_{\alpha}}{q_{\alpha}}_h=\int_{\mathbb{C}^n}  \varphi(z) |q_{\alpha}(z)|^2 e^{-2h(z)}dz \\
           =& \int_{\mathbb{C}^{b_1}}\dots\int_{\mathbb{C}^{b_k}} \, \varphi(z_{(1)},\dots,z_{(k)})|q_{\alpha}(z_{(1)},\dots,z_{(k)})|^2 e^{-2h(|z|)}\, dz_{(1)}\dots dz_{(k)} \\
           =& \int_{\mathbb{S}^{2b_1-1}}\dots\int_{\mathbb{S}^{2b_k-1}} \int_{\mathbb{R}_+^k} \mathbf{a}_{\varphi}(r_1,\dots ,r_k) \prod_{i=1}^k |\omega_j^{\alpha_j}|^2  r_j^{2s_j+2b_j-1}  e^{-2h(|\mathbf{r}|)}\, d\mathbf{r}d\sigma(\boldsymbol{\omega})\\
           =&  \prod_{j=1}^k \dfrac{2\pi^{b_j}\alpha_j!}{(b_j-1+|\alpha_j|)!} \int_{\mathbb{R}_+^k} \mathbf{a}_{\varphi}(r_1,\dots ,r_k) \prod_{j=1}^k r_j^{2s_j+2b_j-1}  e^{-2h(|\mathbf{r}|)}\, d\mathbf{r}
         \end{align*}
         
Since $\innerpro{\lambda_{\varphi,\mathbf{s}}\,q_\alpha}{q_\alpha}_h=\innerpro{\varphi\,q_\alpha}{q_\alpha}_h$. Then we obtain the desired result by dividing by $\innerpro{q_\alpha}{q_\alpha}_h$.

     \end{proof}
\end{theorem}

\section{Density of Toeplitz Operators }
Let $H$ be a compact subgroup of $\U(n)$, and let $\mathcal{T}(\mathcal{A}_h(\mathbb{C}^n))^H$ be the $C^*$-algebra generated by Toeplitz operators with $L^\infty(\mathbb{C}^n)$ symbols that are invariant under $H$. We will show that $\mathcal{T}(\mathcal{A}_h(\mathbb{C}^n))^H$ is dense in $\End_H(\mathcal{A}_h(\mathbb{C}^n))$ under the strong operator topology; this will help us characterize the commutativity of $\mathcal{T}(\mathcal{A}_h(\mathbb{C}^n))^H$ in terms of multiplicity-free representation.

\begin{lemma}[\cite{Englis-Density-of-algebras}]
Let $\Omega$ be a domain in $\mathbb{C}^n$ and $F(x,y)$ be defined on $\Omega\times \bar{\Omega}$ such that $F$ is holomorphic in the first variable and anti-holomorphic in the second variable. If $F(x,\bar{x})=0$ for every $x$ in $\Omega$, then $F$ is identically zero on $\Omega\times \bar{\Omega}$.
\end{lemma}

The following lemma appeared in the proof of Theorem 2 in \cite{Englis-Density-of-algebras} for Toeplitz operators on the Bergman space; we will show that it still holds for the weighted radial Fock space. In fact, Engli\v{s} \cite{Englis-Density-of-algebras} stated that the argument works for the classical Fock space. We will use this lemma to prove that, if $H$ is a compact subgroup of $\U(n)$, then the space of Toeplitz operators with $H$-invariant symbols is dense in $\End_H(\mathcal{A}_h(\mathbb{C}^n))$ with respect to the strong operator topology.

\begin{lemma}
\label{zero on the diagonal}
    Suppose that $\lbrace f_1,\dots,f_q\rbrace$ and $\lbrace g_1,\dots,g_s\rbrace$ are bases of finite-dimensional subspaces of $\mathcal{A}_h(\mathbb{C}^n)$. If $u\in \mathbb{C}^{q\times  s}$ satisfies \[ \sum_{j=1}^s\sum_{i=1}^q \langle T_{\varphi}f_i,g_j\rangle _h \,u_{ij}=0 \] for every $\varphi\in L^{\infty}(\mathbb{C}^n)$, then $u=0$.

    \begin{proof}  
Suppose we have $u\in\mathbb{C}^{q\times s}$ such that \[ \sum_{j=1}^s\sum_{i=1}^q \langle T_{\varphi}f_i,g_j\rangle_h \, u_{ij}=0 \] for every $\varphi \in L^{\infty}(\mathbb{C}^n)$. Then \[ \int_{\mathbb{C}^n} \varphi(z) e^{-2h(z)}\sum_{j=1}^s\sum_{i=1}^q u_{ij}f_i(z)\overline{g_j(z)}\,dz=0 \] for all $\varphi\in L^{\infty}(\mathbb{C}^n)$. Therefore, we have
\begin{equation}
\label{C1}
  \sum_{j=1}^s\sum_{i=1}^q u_{ij}f_i(z)\overline{g_j(z)}=0 
\end{equation} almost everywhere on $\mathbb{C}^n$. Since the left-hand side of equation \ref{C1} is continuous, then the equation holds everywhere on $\mathbb{C}^n$. Define the function \[ F(x,y)=\sum_{j=1}^s\sum_{i=1}^q u_{ij}f_i(x)\overline{g_j(\overline{y})} .\] Then $F(z,\bar{z})=0$ on $\mathbb{C}^n$. Therefore, the function $F$ is zero on $\mathbb{C}^n\times \mathbb{C}^n$. Since the functions $f_i's$ and $g_i's$ are linearly independent, then $u_{ij}=0$ for every $i,j$. 
\end{proof}

\end{lemma}

\begin{theorem}\label{generalized Englis density theorem}
 Let $H$ be a compact subgroup of $\U(n)$. The space $\mathcal{T}(\mathcal{A}_h(\mathbb{C}^n))^{H}$ is dense in $\End_{H}(\mathcal{A}_{h}(\mathbb{C}^n))$, the space of $H$-intertwining operators on $\mathcal{A}_{h}(\mathbb{C}^n)$, under the strong operator topology.
 \end{theorem}
 \begin{proof}
   First we will show that the space $\mathcal{T}(\mathcal{A}_h(\mathbb{C}^n))=\lbrace T_\varphi: \varphi\in L^\infty(\mathbb{C}^n) \rbrace$ is dense in $\mathcal{B}(\mathcal{A}_h(\mathbb{C}^n))$ in the strong operator topology. Let $\lbrace f_1,...,f_q \rbrace$ and $\lbrace g_1,...,g_r \rbrace $ be two linearly independent families in $\mathcal{A}_h(\mathbb{C}^n)$. Let $T$ be a bounded operator on $\mathcal{A}_{h}(\mathbb{C}^n)$, we will show that there exists $\varphi\in L^\infty(\mathbb{C}^n)$ such that \[ \langle Tf_i,g_j\rangle \,\, =\,\, \langle T_{\varphi}f_i,g_j\rangle  \]
 
 Define the operator \[ R:L^\infty(\mathbb{C}^n)\rightarrow \mathbb{C}^{q\times r} \] \[ (R\varphi)_{ij}:=\,\, \langle T_{\varphi}f_i,g_j\rangle_h \]
 Suppose that $u\in\mathbb{C}^{q\times r}$ is orthogonal to the image of the operator $R$, that is   \[ \sum_{j}\sum_{i} \langle T_{\varphi}f_i,g_j\rangle \overline{u_{ij}}=0 \] 
  for every $\varphi\in L^\infty(\mathbb{C}^n)$. Hence, by Lemma \ref{zero on the diagonal} we have $u=0$. In other words, the image of $R$ is $\mathbb{C}^{q\times r}$. 
  
This shows that $\mathcal{T}(\mathcal{A}_h(\mathbb{C}^n))$ is dense in $\mathcal{B}(\mathcal{A}_h(\mathbb{C}^n))$ in the weak operator topology. Since $\mathcal{T}(\mathcal{A}_h(\mathbb{C}^n))$ is a convex subspace of $\mathcal{B}(\mathcal{A}_h(\mathbb{C}^n))$, then $\mathcal{T}(\mathcal{A}_h(\mathbb{C}^n))$ is dense in $\mathcal{B}(\mathcal{A}_h(\mathbb{C}^n))$ in the strong operator topology.

Now, let $T\in \End_H(\mathcal{A}_h(\mathbb{C}^n))$. There is a sequence $\lbrace {\varphi_i}\rbrace_{i\in\mathbb{N}} \subset L^\infty(\mathbb{C}^n)$ such that \[T_{\varphi_i}\longrightarrow T\]
in the strong operator topology.

Let $\pi$ be the representation of $H$ on $\mathcal{A}_h(\mathbb{C}^n)$ which is given by \[(\pi(g)f)(z)=f(g^{-1}.z)\]
Let $dg$ be the normalized Haar measure on H, and define \[\widehat\varphi_i(z)=\int_H \varphi_i(g^{-1}.z)\, dg\]
Then $\widehat\varphi_i\in L^\infty(\mathbb{C}^n)$ and $\widehat\varphi_i$ is $H$-invariant. Hence \[T_{\widehat\varphi_i}\in \mathcal{T}(\mathcal{A}_h(\mathbb{C}^n))^H \]
    We now show that $ T_{\widehat\varphi_i}\longrightarrow T$ in the strong operator topology. Let $g\in H$, by Proposition\ref{Toeplitz and rep}  \[ T_{g.\varphi_i}=\pi(g)\, T_{\varphi_i} \, \pi(g^{-1}).\]

For $f\in \mathcal{A}_h(\mathbb{C}^n)$, we have
\begin{align*}
    (T_{\widehat\varphi_i}\, f)(z)&=\int\limits_{\mathbb{C}^n} \widehat\varphi_i(w)\, f(w)\, \overline{K_z(w)}\, e^{-2h(w)}\, dw \\
    &= \int\limits_{\mathbb{C}^n} \int\limits_H \varphi_i(g^{-1}.w)\, f(w)\, \overline{K_z(w)} \,e^{-2h(w)}\, dw \, dg\\
    &=\int\limits_H (T_{g.\varphi_i}\, f)(z)\, dg\\
    &= \int\limits_H \pi(g)\, T_{\varphi_i} \, \pi(g^{-1}) f(z) \, dg
\end{align*}

Since $T\in \End_H(\mathcal{A}_h(\mathbb{C}^n)$, we have \[ \pi(g)\, T \, \pi(g^{-1}) = T\] for every $g\in H$.

For $f\in\mathcal{A}_h(\mathbb{C}^n)$, we have \begin{align}
    \|{( T_{\widehat\varphi_i}- T)f}\|_h &= \left\|{ \int\limits_H \pi(g)( T_{\varphi_i}- T) \pi(g^{-1})f \, dg}\right \|_h\\
    &\leq \int\limits_H  \left\|{  \pi(g)( T_{\varphi_i}- T) \pi(g^{-1})f}\right \|_h\ \, dg \\    
    \label{upper-of-int}
    &= \int\limits_H  \left\|{  ( T_{\varphi_i}- T) \pi(g^{-1})f}\right \|_h\ \, dg 
\end{align}

Because $T_{\varphi_i}\longrightarrow T$ strongly, then we have \[ T_{\varphi_i} f\longrightarrow Tf \]
 for every $f\in \mathcal{A}_h(\mathbb{C}^n)$.  Hence, for fixed $f\in \mathcal{A}_h(\mathbb{C}^n)$, the sequence $T_{\varphi_i} f$ is bounded. So, \[\sup\limits_{i}\| T_{\varphi_i} f\|_h <\infty\]
By the Uniform Boundedness Principle,  \[ \sup\limits_{i}\| T_{\varphi_i}\| <\infty.\] Hence, there exists $M>0$ such that \[\| T_{\varphi_i}\|<M \quad\text{for every }i\in\mathbb{N}. \] 

Therefore, \[ \| T_{\varphi_i}-T\|\leq \| T_{\varphi_i}\| +\|T\| \leq M+\|T\|\]

and this implies \[ \left\|{  ( T_{\varphi_i}- T) \pi(g^{-1})f}\right \|_h \leq \left(M+\|T\|\right) \|f\|_h.\]

For every fixed $g\in H$ we have \[\left\|{  ( T_{\varphi_i}- T) \pi(g^{-1})f}\right \|_h \longrightarrow 0.\]
Consequently, inequality \ref{upper-of-int} and the dominated convergence theorem implies that \[ \|{( T_{\widehat\varphi_i}- T)f}\|_h \longrightarrow 0 .\]

Thus, the sequence $T_{\widehat\varphi_i} $ converges to $T$ in the strong operator topology. This shows that \[ \mathcal{T}(\mathcal{A}_h(\mathbb{C}^n))^H\] is dense in $\End_H(\mathcal{A}_h(\mathbb{C}^n))$ under the strong operator topology.
  \end{proof}

As an application of the density theorem established above, we obtain a characterization of the commutativity of the $H$-invariant $C^*$-algebras $\mathcal{T}(\mathcal{A}_h(\mathbb{C}^n))^H$ in terms of the multiplicity-free representation $\pi|_H$.  

First, we recall the following lemma.

\begin{lemma}[\cite{Dawson2015}]
    \label{commutitivity of dense subspace}
    Suppose that $V$ is a linear subspace of the space $\mathcal{L}(\mathcal{H})$ of bounded linear operators on a Hilbert space $\mathcal{H}$. If $V$ consists of operators that commute, then the closure of $V$ in the strong operator topology also consists of operators which commute.
\end{lemma}

\begin{theorem}
\label{multiplicity free and commutativity}

 If $H$ is a compact subgroup of $\U(n)$, then $\mathcal{T}(\mathcal{A}_h(\mathbb{C}^n))^H$ is commutative if and only if $\pi|_{H}$ is multiplicity-free.
 
\end{theorem}
\begin{proof}
    
Suppose first that the $C^*$-algebra $\mathcal{T}(\mathcal{A}_h(\mathbb{C}^n))^H$ is commutative. By Theorem \ref{generalized Englis density theorem} and Lemma \ref{commutitivity of dense subspace}, it follows that \[\End_H(\mathcal{A}_h(\mathbb{C}^n))\] is commutative. Therefore, the representation $\pi|_{H}$  is multiplicity-free.

Conversely, suppose that $\pi|_{H}$ is multiplicity-free. Then \[ \End_H(\mathcal{A}_h(\mathbb{C}^n))\] is commutative. Since $\mathcal{T}(\mathcal{A}_h(\mathbb{C}^n))^H\subset \End_H(\mathcal{A}_h(\mathbb{C}^n))$, it follows that $\mathcal{T}(\mathcal{A}_h(\mathbb{C}^n))^H$ is also commutative.
\end{proof}

\section{ Center-invariant Toeplitz algebras and non-commutativity}

In this section, we will take a deeper look at the $C^*$-algebra generated by Toeplitz operators with symbols invariant under the center of $\mathbf{U_m}$. As in the previous sections, we fix a partition $\mathbf{m}=(b_1,...,b_k)$ of an integer $n=b_1+...+b_k$. Consider the unitary representation $\pi:=\pi_{b1}\otimes...\otimes \pi_{b_k}$, see equation (\ref{representation of U_m}), of the group $\mathbf{U_m}$. Following the notation in \cite{quiroga2021toeplitz}, we denote by $\I_{(j)}$ the identity operator on $\mathbb{C}^{b_j}$. Consider the group \[\mathbb{T}^{\mathbf{m}}:=\left\lbrace \left(t_1\I_{(1)},...,t_k\I_{(k)}\right): t_i\in \mathbb{T }\right\rbrace\]
which is the center of $\mathbf{U_m}$.
The following result provides the isotypic decomposition of the restriction of the representation $\pi$ to the subgroup $\mathbb{T}^m$.
\begin{theorem}
\label{Isotypic decompisition of T^m}
    Let $\mathbf{m}=(b_1,...,b_k)$ be a partition of an integer $n=b_1+...+b_k$. The isotypic decomposition of $\pi|_{\mathbb{T}^{\mathbf{m}}}$ is
    \[\mathcal{A}_h(\mathbb{C}^n)= \bigoplus_{s_1,...,s_k=0}^{\infty} P^{s_1}(\mathbb{C}^{b_1})\otimes...\otimes P^{s_k}(\mathbb{C}^{b_k}) \]
Moreover, the representation $\pi|_{\mathbb{T}^{\mathbf{m}}}$ is multiplicity-free if and only if $\mathbf{m}=(1,...,1)$.
    \begin{proof}
        The decomposition of $\mathcal{A}_h(\mathbb{C}^n)$ was obtained in Proposition \ref{Isotypic decomposition of U_m}. It remains to determine the action of $\mathbb{T}^{\mathbf{m}}$ on each summand. Let $p(z)\in P^{s_1}(\mathbb{C}^{b_1})\otimes...\otimes P^{s_k}(\mathbb{C}^{b_k})$, and let \[A:=\left(t_1\I_{(1)},...,t_k\I_{(k)}\right)\in\mathbb{T}^{\mathbf{m}}\] Then we have
        \begin{align*}
            A.p(z)=p(A^{-1}.z) &=p\left(t^{-1}_{1}z_{(1)},...,t^{-1}_{k}z_{(k)}\right)\\
            &=t^{-s_1}_1...t^{-s_k}_k p(z)
        \end{align*}

        Hence, $\mathbb{T}^{\mathbf{m}}$ acts on $P^{s_1}(\mathbb{C}^{b_1})\otimes...\otimes P^{s_k}(\mathbb{C}^{b_k})$ through the character 
        \[\chi_{\mathbf{s}}(t_1,...,t_k)=t_1^{-s_1}\dots t_k^{-s_k}\]

        Let $\mathbf{s}=(s_1,\dots,s_k)$ and $\mathbf{r}=(r_1,\dots,r_k)$. If $\mathbf{s}\neq \mathbf{r}$, then $\chi_{\mathbf{s}}\neq \chi_{\mathbf{r}}$. Therefore, the summands \[ P^{s_1}(\mathbb{C}^{b_1})\otimes...\otimes P^{s_k}(\mathbb{C}^{b_k}) \] are the isotypic components of $\pi|_{\mathbf{m}}$.

       Furthermore, the representation $\pi|_{\mathbf{m}}$ is multiplicity-free if and only if $\dim P^{s_j}(\mathbb{C}^{b_j})=1$, and this holds if and only if $b_j=1$ for all $j=1,...,k$.
    \end{proof}
\end{theorem}

For $\mathbf{s}=(s_1,\dots,s_k)\in\mathbb{N}_0^k$, let \[ V_{\mathbf{s}}:=P^{s_1}(\mathbb{C}^{b_1})\otimes\dots\otimes P^{s_k}(\mathbb{C}^{b_k}) \]
An operator $T\in\mathcal{B}(\mathcal{A}_h(\mathbb{C}^n))$ belongs to $\End_{\mathbb{T}^{\mathbf{m}}}(\mathcal{A}_h(\mathbb{C}^n))$ if and only if \[T(V_{\mathbf{s}})\subseteq V_{\mathbf{s}}\]
for every $\mathbf{s}\in\mathbb{N}_0^k$. Therefore, \[\End_{\mathbb{T}^{\mathbf{m}}}(\mathcal{A}_h(\mathbb{C}^n))= \left\lbrace \bigoplus_{\mathbf{s}\in\mathbb{N}_0^k}\, T_{\mathbf{s}}: T_{\mathbf{s}}\in\mathcal{B}(V_\mathbf{s}) \text{ and } \sup_{\mathbf{s}\in\mathbb{N}_0^k} \|T_{\mathbf{s}}\|<\infty \right\rbrace\]

Moreover, if $\mathbf{m}=(1,\dots,1)$ is a partition of an integer $n$, then $\dim(V_{\mathbf{s}})=1$ for every $\mathbf{s}\in\mathbb{N}_0^n$. Hence every $T\in \End_{\mathbb{T}^{\mathbf{m}}}(\mathcal{A}_h(\mathbb{C}^n)) $ acts on $V_{\mathbf{s}}$ as multiplication by a scalar. Therefore, \[\End_{\mathbb{T}^{\mathbf{m}}}(\mathcal{A}_h(\mathbb{C}^n)) \simeq \ell^\infty(\mathbb{N}_0^n)\]
In fact, this shows that if $\mathbf{m}=(1,...,1)$, then $\End_{\mathbb{T}^{\mathbf{m}}}(\mathcal{A}_h(\mathbb{C}^n))$ is commutative. Consequently, the $C^*$-algebra $\mathcal{T}(\mathcal{A}_h(\mathbb{C}^n))^{\mathbb{T}^{\mathbf{m}}}$  is also commutative. We record this result in the following proposition, where the proof follows from Theorems \ref{multiplicity free and commutativity} and \ref{Isotypic decompisition of T^m}. 

\begin{proposition}

    The $C^*$-algebra $\mathcal{T}(\mathcal{A}_h(\mathbb{C}^n))^{\mathbb{T}^{\mathbf{m}}}$  is commutative if and only if $\mathbf{m}= (1,...,1)$.
    
\end{proposition}

Since $\mathbb{T}^{\mathbf{m}}$ is a compact subgroup of $\mathbf{U_m}$, by Theorem \ref{generalized Englis density theorem} we obtain the following result.
\begin{proposition}
    The $C^*$-algebra $\mathcal{T}(\mathcal{A}_h(\mathbb{C}^n))^{\mathbb{T}^{\mathbf{m}}}$ is dense in $\End_{\mathbb{T}^{\mathbf{m}}}(\mathcal{A}_h(\mathbb{C}^n))$ under the strong operator topology.
\end{proposition}

\section{Failure of the Square-Root Metric Description}

Let $h(r)=\frac{1}{2}r^2$, then $\mathcal{A}_h(\mathbb{C})$ is the classical Fock space. It is well known that the $C^*$-algebra generated by Toeplitz operators with bounded radial  symbols is isometrically isomorphic to the $C^*$-algebra generated by the set $\mathfrak{B}$ of all sequences of the eigenvalues \[ \mathfrak{B}:= \lbrace \lambda_{a}: a\in L^\infty(\mathbb{R}^+) \rbrace\]
Esmeral and Maximenko \cite{esmeral2016radial} proved that the uniform closure of $\mathfrak{B}$ coincides with the $C^*$-algebra consisting of all bounded sequences that are uniformly continuous with respect to the square-metric \[ \rho(m,n)=|\sqrt{m}-\sqrt{n}|,\quad m,n\in \mathbb{Z}^+\]
This result was generalized by Dewage and \'{O}lafsson \cite{dewage2022toeplitz} to the $k$-quasi-radial symbols on the classical Fock space $\mathcal{F}^2(\mathbb{C}^n)$.

In this section, we will show that this does not hold when considering general radial weights. In particular, we will consider the logarithmically growing weight

 \[ h(r)=\begin{cases}
    0 & \,: 0\leq r \leq 1\\[2mm]
    \frac{1}{2}(\log r)^{3/2} &:\, 1<r
\end{cases}\]

This kind of weights appears in the theory of small Fock spaces; see, for example, \cite{borichev2010riesz, kellay2020riesz}.

Now we state the main theorem in this section:

\begin{theorem} \label{not-uniformly-cont}

For the radial weight
\[ h(r)=\begin{cases}
    0 & \,: 0\leq r \leq 1\\[2mm]
    \frac{1}{2}(\log r)^{3/2} &:\, 1<r
\end{cases}\]
there exists a bounded radial symbol $\varphi$ such that the eigenvalue
sequence
\[
\bigl(\lambda_{\varphi,m}\bigr)_{m\in\mathbb N_0}
\]
is not uniformly continuous with respect to the square-root metric
\[
\rho(m,n)=|\sqrt{m}-\sqrt{n}|.
\]
\end{theorem}

To prove Theorem \ref{not-uniformly-cont}, we will need a technical lemma. First, it is clear that $h(r)$ is continuous on $[0,\infty)$ and \[\int\limits_0^\infty r^{2m+1}\,e^{-2h(r)}<\infty\]

 for every $m\in\mathbb{N}_0$.

Write
\begin{align*}
    \int\limits_0^\infty r^{2m+1}\,e^{-2h(r)}dr&=\int\limits_0^1 r^{2m+1}\,dr+ \int\limits_1^\infty r^{2m+1}\,e^{-(\log r)^{3/2}}\,.dr\\
    &=\int\limits_0^1 r^{2m+1}\,dr + \int\limits_0^\infty e^{(2m+2)u-u^{3/2}}\,du 
\end{align*}

and let $F_m(u)=(2m+2)u-u^{3/2}$. The function $F_m$ has a maximum value at \[u_m:=\dfrac{4}{9}(2m+2)^2\]

Let $R_{-1}=0$ and $R_m=\exp{\left[\dfrac{4}{9}(2m+3)^2\right]}$. Hence, the point $u_m$ at which $F_m$ has a maximum value of lies in the interval $\big(\log(R_{m-1}), \log(R_m)\big) $.

Define the function $\varphi$ on $\mathbb{C}$ by \[\varphi(z)=(-1)^m \text{ whenever } R_{m-1}\leq |z|< R_m .\] The function $\varphi$ is bounded and radial. Consider the Toeplitz operator $T_\varphi$. By Theorem \ref{eigenvalues radial}, we have \[T_\varphi|_{P^m(\mathbb{C})}=\lambda_{\varphi,m}\I\]
where the eigenvalues $\lambda_{\varphi,m}$ are given by  \[\lambda_{\varphi,m}=\dfrac{\displaystyle\int_0^\infty \varphi(u)\,u^{2m+1}e^{-2h(u)}\, du}{\displaystyle\int_0^\infty\, u^{2m+1}e^{-2h(u)}\,du}\]

\begin{lemma}
\label{alternating-eigenvalues}
For the symbol $\varphi$ defined above,
\[
\lim_{m\rightarrow\infty}\left|\lambda_{\varphi,m}-(-1)^m\right|=0
\]

\begin{proof}
 Let
\[
A_m:=\int_0^\infty u^{2m+1}e^{-2h(u)}\,du.
\]
Then
\[
\left|\lambda_{\varphi,m}-(-1)^m\right|
\leq
\dfrac{1}{A_m}
\int_0^\infty
\left|\varphi(u)-(-1)^m\right|
u^{2m+1}e^{-2h(u)}\,du
\]
Since $\varphi(u)=(-1)^m$ on $[R_{m-1},R_m)$, we obtain
\[
\left|\lambda_{\varphi,m}-(-1)^m\right|
\leq
\dfrac{2B_m+2C_m}{A_m}
\]
where
\[
B_m:=\int_0^{R_{m-1}}u^{2m+1}e^{-2h(u)}\,du
\]
and
\[
C_m:=\int_{R_m}^{\infty}u^{2m+1}e^{-2h(u)}\,du
\]
We will show that
\[
\frac{B_m}{A_m}\longrightarrow0
\qquad\text{and}\qquad
\frac{C_m}{A_m}\longrightarrow0
\]
For $m\in\mathbb{N}_0$ large enough,
\[
B_m
=
\int_0^1u^{2m+1}\,du
+
\int_0^{\log(R_{m-1})}e^{F_m(u)}\,du
\]
We first obtain a lower bound for $A_m$. Since
\[
F_m''(u)=-\dfrac34u^{-1/2}
\]
 Taylor's theorem implies that there exists  $\xi\in[u_m,u]$ such that
\begin{align*}
F_m(u)
&=F_m(u_m)+F_m'(u_m)(u-u_m)+\frac{1}{2}F_m''(\xi)(u-u_m)^2\\
&=F_m(u_m)-\frac{3}{8}\xi^{-1/2}(u-u_m)^2\\
&\geq F_m(u_m)-1
\end{align*}
 Therefore,
\begin{equation}
\label{A_mlowerbound}
A_m\geq\int_{u_m}^{u_m+1}e^{F_m(u)}\,du
\geq e^{F_m(u_m)-1}
\end{equation}
Since $A_m\rightarrow\infty$, we have 
\[
\frac{1}{A_m}\int_0^1u^{2m+1}\,du
=
\frac{1}{(2m+2)A_m}
\longrightarrow0
\]
Since $F_m(u)$ is increasing on $[0,\log(R_{m-1})]$, we have
\begin{equation}
\label{R_m-1upperbound}
\int_0^{\log(R_{m-1})}e^{F_m(u)}\,du
\leq
\log(R_{m-1})
\exp\left[F_m\left(\log(R_{m-1})\right)\right]
\end{equation}
Using \eqref{R_m-1upperbound} and \eqref{A_mlowerbound}, we obtain
\begin{align*}
\dfrac{1}{A_m}\int_0^{\log(R_{m-1})}e^{F_m(u)}\,du
&\leq
e\log(R_{m-1})
\exp\left[
F_m(\log(R_{m-1}))-F_m(u_m)
\right]\\
&=
e\frac{4}{9}(2m+1)^2
\exp\left[
F_m\left(\frac{4}{9}(2m+1)^2\right)
-
F_m\left(\frac{4}{9}(2m+2)^2\right)
\right]\\
&=
e\frac{4}{9}(2m+1)^2
\exp\left[-\frac{8}{27}(3m+2)\right]
\longrightarrow0
\end{align*}
Therefore,
\[
\frac{B_m}{A_m}\longrightarrow0
\]
It remains to show that
\[
\frac{C_m}{A_m}\longrightarrow0
\]
By the change of variables $u=\log r$,
\[
C_m
=
\int_{\log(R_m)}^\infty e^{F_m(u)}\,du
\]
Since
\[
F_m'(u)=(2m+2)-\frac32\sqrt{u}
\]
then $F_m'$ is decreasing on $[\log(R_m),\infty)$. If $u\geq\log(R_m)$, then 
\[
F_m'(u)\leq F_m'(\log(R_m))=-1.
\]
Thus,
\[
F_m(u)-F_m(\log(R_m))
=
\int_{\log(R_m)}^uF_m'(t)\,dt
\leq
\log(R_m)-u
\]
and therefore
\[
e^{F_m(u)}
\leq
e^{F_m(\log(R_m))}
e^{\log(R_m)-u}
\]
Integrating over $[\log(R_m),\infty)$ gives
\begin{equation}
\label{R_m-upper-bound}
\int_{\log(R_m)}^\infty e^{F_m(u)}\,du
\leq
e^{F_m(\log(R_m))}
\end{equation}
Using \eqref{A_mlowerbound} and \eqref{R_m-upper-bound}, we obtain
\begin{align*}
\dfrac{C_m}{A_m}
&\leq
e\exp\left[
F_m(\log(R_m))-F_m(u_m)
\right]\\
&=
e\exp\left[
F_m\left(\frac{4}{9}(2m+3)^2\right)
-
F_m\left(\frac{4}{9}(2m+2)^2\right)
\right]\\
&=
e\exp\left[-\frac{8}{27}(3m+4)\right]
\longrightarrow0
\end{align*}
Hence
\[
\left|\lambda_{\varphi,m}-(-1)^m\right|
\leq
2\frac{B_m}{A_m}
+
2\frac{C_m}{A_m}
\longrightarrow0
\]
and this proves the lemma.

\end{proof}

\end{lemma}

Now we are ready to prove Theorem \ref{not-uniformly-cont}.

\begin{proof} 

By Lemma~\ref{alternating-eigenvalues},
\[
\left|\lambda_{\varphi,m}-(-1)^m\right|\longrightarrow0
\]
Therefore,
\[
|\lambda_{\varphi,m+1}-\lambda_{\varphi,m}|
\longrightarrow 2
\]
On the other hand,
\[
\rho(m+1,m)
=
\sqrt{m+1}-\sqrt{m}
\longrightarrow0
\]
Hence the sequence $\left(\lambda_{\varphi,m}\right)_{m\in\mathbb{N}_0}$ is not uniformly continuous with respect to the square-root metric $\rho$.
    
\end{proof}

\begin{corollary}
The uniform closure of the set of eigenvalue sequences
\[
\left\{
(\lambda_{\varphi,m})_{m\in\mathbb{N}_0} :
\varphi\in L^\infty(\mathbb{C}) \text{ is radial}
\right\}
\]
does not coincide with the $C^*$-algebra of all bounded sequences that are uniformly continuous with respect to the square-root metric
\[
\rho(m,n)=|\sqrt{m}-\sqrt{n}|.
\]
\end{corollary}

\bibliographystyle{amsplain}
\bibliography{sn-bibliography}

\end{document}